\documentclass[12pt,twoside]{amsart}
\usepackage{amsmath,bm}
\usepackage{amscd}
\usepackage{amsmath}
\usepackage{latexsym}
\usepackage{amsfonts}
\usepackage{amssymb}
\usepackage{amsthm}
\usepackage{graphicx}
\usepackage{verbatim}
\usepackage{mathrsfs}
\usepackage{enumerate}
\usepackage{color}
\usepackage{titlesec}
\usepackage{multirow}
\usepackage{booktabs}
\titleformat{\section}[block]{\normalsize\bfseries}{\arabic{section}.}{1em}{ }[]
\titleformat{\subsection}[block]{\normalsize\bfseries}{\arabic{section}.\arabic{subsection}}{1em}{}[]
\titleformat{\subsubsection}[block]{\normalsize\bfseries}{\arabic{subsection}-\alph{subsubsection}}{1em}{}[]
\titleformat{\paragraph}[block]{\small\bfseries}{[\arabic{paragraph}]}{1em}{}[]

\numberwithin{equation}{section}
\newtheorem{theorem}{Theorem}[section]
\newtheorem{definition}[theorem]{Definition}

\newtheorem{lemma}[theorem]{Lemma}
\newtheorem{corollary}[theorem]{Corollary}
\newtheorem{proposition}[theorem]{Proposition}

\newtheorem{remark}[theorem]{Remark}

\newcommand{\ZZ}{\mathbb{Z}}

\newcommand{\RR}{\mathbb{R}}

\newcommand{\dd}{\,\mathrm{d}}

\newcommand{\Bppp}{ \dot B^{\frac{d}{r}+3}_{r,1}}
\newcommand{\Bp}{\dot  B^{\frac{d}{r}}_{r,1}}
\newcommand{\Bpp}{ \dot B^{\frac{d}{r}+2}_{r,1}}
\newcommand{\Bpo}{ \dot B^{\frac{d}{r}+1}_{r,1}}
\newcommand{\B}{ \dot B^{\frac{d}{r}-1}_{r,1}}
\DeclareMathOperator{\Div}{div}

\makeatletter

\newcommand{\Rmnum}[1]{\expandafter\@slowromancap\romannumeral #1@}
\makeatother
\allowdisplaybreaks
\begin{document}
\title[]{Sharp ill-posedness for
the non-resistive MHD equations in  Sobolev spaces}
\author[Q. Chen, Y. Nie, W. Ye]{Qionglei Chen, Yao Nie, Weikui Ye}
\address{Institute of Applied Physics and Computational
Mathematics, Beijing 100191, P.R. China}
\email{chen\_qionglei@iapcm.ac.cn}
\address{School of Mathematical Sciences and LPMC, Nankai University, Tianjin,
300071, P.R.China}
\email{nieyao@nankai.edu.cn}
\address{ School of Mathematical Sciences, South China Normal University
Guangzhou, Guangdong, 510631, P. R. China}
\email{904817751@qq.com}

\keywords{The non-resistive MHD equations; Ill-posedness; Besov spaces.}
\maketitle\title{\vspace{-1em}}

\begin{abstract}
In this paper, we  prove a sharp ill-posedness result for the incompressible non-resistive MHD equations. In any dimension $d\ge 2$, we show the ill-posedness of the non-resistive MHD equations in $H^{\frac{d}{2}-1}(\RR^d)\times H^{\frac{d}{2}}(\RR^d)$, which is sharp in view of the results of the local well-posedness in $H^{s-1}(\RR^d)\times H^{s}(\RR^d)(s>\frac{d}{2})$ established by  Fefferman et al.(Arch. Ration. Mech. Anal., \textbf{223}
(2), 677-691, 2017). Furthermore, we generalize the ill-posedness results from $H^{\frac{d}{2}-1}(\RR^d)\times H^{\frac{d}{2}}(\RR^d)$ to Besov spaces $B^{\frac{d}{p}-1}_{p, q}(\RR^d)\times B^{\frac{d}{p}}_{p, q}(\RR^d)$ and $\dot B^{\frac{d}{p}-1}_{p, q}(\RR^d)\times \dot B^{\frac{d}{p}}_{p, q}(\RR^d)$ for $1\le p\le\infty, q>1$.  Different from the ill-posedness mechanism of the incompressible Navier-Stokes equations in $\dot B^{-1}_{\infty, q}$ \cite{B,W}, we construct an initial data such that     the { paraproduct terms} (low-high frequency interaction) of the nonlinear term  make the main contribution to the norm inflation of the magnetic field. 
\end{abstract}
\section{Introduction}
{Magneto-hydrodynamics (MHD) is concerned with the study of the mutual interaction between magnetic fields and electrically conducting
fluids}. In this paper, we investigate the Cauchy problem for  the incompressible non-resistive MHD equations in $\RR^d$ for $d\ge 2$:
\begin{equation}\label{MHD}
\left\{ \aligned
    & \partial_t u-\Delta u+\nabla P=b\cdot\nabla b-u\cdot\nabla u,\\
     &\partial_t b+u\cdot\nabla b=b\cdot\nabla u,\\
     &\Div u=\Div b=0,\\
     &(u,b)|_{t=0}=(u_0, b_0).
\endaligned
\right.
\end{equation}
Here the initial data $(u_0, b_0)$ is divergence-free, $u=(u^1, u^2, \cdots, u^d)$ represents the velocity field, $b=(b^1, b^2, \cdots, b^d)$ denotes the magnetic field and $P$ is the scalar pressure. The non-resistive MHD system \eqref{MHD} can {be applied to} describe strong collisional plasmas, or plasmas with  extremely small resistivity due to these collisions \cite{C70, Lin}.

Compared with the {viscous and resistive  MHD system}, the study of well-posedness of the non-resistive MHD system \eqref{MHD}  becomes much more difficult owning to the hyperbolic type of magnetic equation. In recent years, there {has been}  some  significant progress in the global well-posedness of the system \eqref{MHD} with smooth initial data that is close to some nontrivial steady state (see \cite{AZ, LXZ, RWXZ, XZ, ZT} and the references therein). For the local-in-time existence of solutions to the non-resistive MHD system \eqref{MHD}, Jiu and Niu \cite{JN} firstly showed the local well-posedness result in $H^s(\RR^2)\times H^s(\RR^2)$ with $s\ge 3$ via viscous approximations. By means of a new commutator estimate, Fefferman et al. \cite{FMRR} proved  the local-in-time existence and uniqueness of strong solutions in $H^s(\RR^d)\times H^s(\RR^d)$ for $s>\frac{d}{2}, d=2,3$. Later, relying on maximal regularity estimates for the Stokes equation, {the authors in} \cite{FMRR2} established the local-in-time existence and uniqueness of solutions in $ H^{s-1+\varepsilon}(\RR^d)\times H^{s}(\RR^d)$ for $s>\frac{d}{2}, 0<\varepsilon<1, d=2, 3$. {Indeed, via time-space mixed Besov spaces $L^1_T(B^{\frac{d}{2}+1}_{2,1})$, one can generalize well-posedness result to  $ H^{s-1}(\RR^d)\times H^{s}(\RR^d)$ for $s>\frac{d}{2}$.} 
 With respect to  Besov spaces,  Chemin et al. \cite{CMRR} made generalisation of the main result in \cite{FMRR} and obtained the local existence with initial data $(u_0, b_0)\in B^{\frac{d}{2}-1}_{2,1}(\RR^d)\times B^{\frac{d}{2}}_{2,1}(\RR^d)$, $d=2, 3$.  Meanwhile, {the authors in} \cite{CMRR} proposed an interesting problem whether  or not the solution for the Cauchy problem of the system \eqref{MHD} exists locally in time and is unique in corresponding homogeneous Besov spaces.
Subsequently, Li, Tan and Yin \cite{LTY} solved this problem by showing  the local existence and uniqueness of the solution in $\dot B^{\frac{d}{p}-1}_{p, 1}(\RR^d)\times \dot B^{\frac{d}{p}}_{p, 1}(\RR^d)$ for $1\le p\le 2d$ and $d\ge 2$. Ye, Luo and Yin \cite{YLY} generalized the local existence result in $\dot B^{\frac{d}{p}-1}_{p, 1}(\RR^d)\times \dot B^{\frac{d}{p}}_{p, 1}(\RR^d)$ from $1\le p\le 2d$ to $1\le p<\infty$ and proved that the solution map from the initial data $(u_0, b_0)$ to solution $(u, b)$ is continuous from $\dot B^{\frac{d}{p}-1}_{p, 1}\times \dot B^{\frac{d}{p}}_{p, 1}$ to $C([0, T]; \dot B^{\frac{d}{p}-1}_{p, 1}\times \dot B^{\frac{d}{p}}_{p, 1})$ for $d\ge2, 1\le p\le 2d$, which combined with the result in \cite{LTY} shows the local well-posedness of the system \eqref{MHD} in $\dot B^{\frac{d}{p}-1}_{p, 1}\times \dot B^{\frac{d}{p}}_{p, 1}$ if $1\le p\le 2d$ for $d\ge 2$. Unfortunately, { whether  the system \eqref{MHD} is well-posed or not in $\dot B^{\frac{d}{p}-1}_{p, 1}\times \dot B^{\frac{d}{p}}_{p, 1} (2d<p<\infty)$  is still open.} 

When $b=0$, the system \eqref{MHD} is reduced to the classical incompressible Navier-Stokes  equations. Well-posedness issues of the Navier-Stokes equations in different types of the critical spaces have attracted attention of many researchers and there have been {many }relevant results until now (e.g. \cite{B,Can,Che,FK64,Kato84,KT,P,W,Y}). In the context of the critical Besov spaces $\dot B^{\frac{d}{p}-1}_{p, q}(p<\infty, q\le\infty)$,  the well-posedness of global strong solution with small initial data was established by Planchon \cite{P}, Cannone \cite{Can} and Chemin \cite{Che}. By showing the solution map is
 discontinuous at origin,  Bourgain-Pavlovi\'{c} \cite{B}, Yoneda \cite{Y} and Wang \cite{W} verified the ill-posedness of the Navier-Stokes equations in $\dot B^{-1}_{\infty, q}(1\le q\le\infty)$. {These results imply} the ill-posedness of the non-resistive MHD system \eqref{MHD} in $\dot B^{-1}_{\infty, q}\times \dot B^{0}_{\infty, q}$ for $q\ge 1$.

{Throughout the current literature,  Fefferman et al. in \cite{FMRR} established the local-in-time existence and uniqueness of strong solutions in $H^{s}(\RR^d)$ for $s>\frac{d}{2}$ to the system \eqref{MHD}  and suspected that it seems  ill-posed in $H^{\frac{d}{2}}(\RR^d)$. Subsequently,  they in  \cite{FMRR2} proved that  this system is local well-posedness in $H^{s+\epsilon-1}(\RR^d)\times H^s(\RR^d)$ for $s>\frac{d}{2}$ and concluded that this result is nearly optimal  in the scale of Sobolev spaces. In view of their interest on the sharp well-posedness result of this system in Sobolev spaces, we are focused on the problem \emph{whether this system is well-posed or not in $H^{\frac{d}{2}-1}(\RR^d)\times H^{\frac{d}{2}}(\RR^d)$}. Furthermore,  the problem whether the non-resistive MHD system~\eqref{MHD} is well-posed or not in $\dot B^{\frac{d}{p}-1}_{p,q}\times \dot B^{\frac{d}{p}}_{p,q}$ if $1\le p <\infty, q>1$ remains unsolved.  Therefore, we are interested in the following question:}

\noindent\textbf{Question:}\\
\quad\quad\emph{Is the system \eqref{MHD} well-posed or ill-posed in $H^{\frac{d}{2}-1}(\RR^d)\times H^{\frac{d}{2}}(\RR^d) $ and $\dot B^{\frac{d}{p}-1}_{p,q}(\RR^d)\times \dot  B^{\frac{d}{p}}_{p,q}(\RR^d)$ for $1\le p <\infty, q>1$?}

In the important work \cite{B} of  Bourgain-Pavlovi\'{c}, they  showed that the worst contribution
on regularity of solution to the Navier-Stokes equations in a short time stems from the {remainder terms} (high-high frequency interaction), which is a part of the Bony decomposition of the convection term $u\cdot\nabla u$. This  core idea that { remainder } terms in the convection term {prevent from the continuity of solution mapping } has also been applied to other ill-posedness results of Navier-Stokes equations \cite{W, Y}.

Different from the mechanism of the Navier-Stokes equations in Besov spaces, it is the\emph{ paraproduct terms} (low-high frequency interaction) not the {remainder }terms of the nonlinear term $b\cdot \nabla u$ of the system \eqref{MHD} that may lead to the discontinuous solution map of the magnetic field $b$ in the context of $ \dot B^{\frac{d}{p}}_{p,q}$. This observation forces us to construct an example to saturate the paraproduct operator from $\dot B^{\frac{d}{p}}_{p, 1}\times \dot B^{\frac{d}{p}}_{p,q}$ to $\dot B^{\frac{d}{p}}_{p,q}$ {for $q>1$}.
{More precisely, the key ingredient is to construct two Schwartz functions  $f, g$ such that
\[\big\| \mathscr{F}^{-1}\big(\widehat{f}_{2^{\frac{N}{2}}\leq|\xi|\leq2^{\frac{4N}{5}}}\big)
\mathscr{F}^{-1}\big(\widehat{g}_{2^{N-1} \leq|\xi|\leq2^{N+1}}\big)\big\|_{\dot B^{\frac{d}{p}}_{p,q}}\sim\|f
\|_{\dot B^{\frac{d}{p}}_{p,1}}\| g\|_{{\dot B^{\frac{d}{p}}_{p,q}}}.\]
Furthermore,  $f,g$ {satisfy} the following  estimates: for $q>1$ and $\frac{1}{q}<\alpha<1$,
$$\|f\|_{\dot B^{\frac{d}{p}}_{p,q}}+\|g\|_{\dot B^{\frac{d}{p}}_{p,q}}\leq (\ln\ln N)^{-1},\,\,\,\|f\|_{ \dot B^{\frac{d}{p}}_{p,1}}\sim N^{1-\alpha}, \,\,\,\|f\|_{ \dot B^{\frac{d}{p}}_{p,1}}\|g\|_{\dot B^{\frac{d}{p}}_{p,q}}\sim (\ln\ln N)^{-1}N^{1-\alpha}.$$
{Based on  such scalar functions $f$ and $g$, we construct a special initial data $b_0$ whose frequency is supported in a family of  cuboids in different dyadic annuli. Distinct from the structure of $b_0$, the frequency of $u_0$ is higher and locates in one cuboid  with different direction from that of $b_0$.  By using the Lagrangian coordinates in the equation of the magnetic field and the asymmetric structure between $u_0$ and $b_0$, we can show that  the frequency of the second approximation concentrates on  the superposition of many cuboids and thereby the norm inflation phenomenon occurs for a short time.}

}

{Firstly, let us recall a  the local-in-time existence and uniqueness of strong solutions in $H^s$ for $s>\frac{d}{2}$ to the  non-resistive MHD equations (1.1).
\begin{theorem}[\cite{FMRR2}]Let $d\ge2$. Take  $s>\frac{d}{2}$ and $0<\epsilon<1$.  Suppose that the initial
conditions satisfy $u_0\in H^{s-1+\epsilon}(\RR^d)$ and $b_0\in H^{s}(\RR^d)$. Then there exists $T_{\ast}>0$
such that the non-resistive MHD system \eqref{MHD} has a unique solution  $(u, b)$ with $b\in C([0, T_{\ast}); H^{s}(\RR^d))$ and
\[u\in C([0, T_{\ast}); H^{s-1+\epsilon}(\RR^d))\cap L^2([0, T_{\ast}); H^{s+\epsilon}(\RR^d))\cap L^1([0, T_{\ast}); H^{s+1}(\RR^d)).\]
\end{theorem}
\begin{remark}By using the time-space mixed Besov spaces $\mathcal{L}^1_T(B^{\frac{d}{2}+1}_{2,1})$, one can show that the system \eqref{MHD} is locally well-posed in $H^{s-1}(\RR^d)\times H^s(\RR^d)$ for $s>\frac{d}{2}$. That is, let $d\ge2$ and $s>\frac{d}{2}$, if initial data $(u_0, b_0)\in H^{s-1}(\RR^d)\times H^s(\RR^d)$, there exists $T_{\ast}>0$ such that the  system \eqref{MHD} has a unique solution  $(u, b)$ with
$u\in C([0, T_{\ast}); H^{s-1}(\RR^d))\cap \mathcal{L}^1([0,T_{\ast}); B^{s+1}_{2,2})$ and $b\in C([0, T_{\ast}); H^{s}(\RR^d))$.
\end{remark}
This theorem guarantees that for any initial data $(u_0, b_0)\in \mathcal{S}(\RR^d)$, there locally exists a unique solution $(u,b)$ to  the  non-resistive MHD equations \eqref{MHD}.}

{Next, we are in position to state our ill-posedness result by constructing special Schwarz functions as initial data. }
\begin{theorem}\label{H}Let $d\ge 2$.  The system \eqref{MHD} is ill-posed in $H^{\frac{d}{2}-1}(\RR^d)\times H^{\frac{d}{2}}(\RR^d)$ in the sense of ``norm inflation''. More precisely, for any $\delta>0$, there exists a solution $(u, b)$ to the non-resistive MHD equations \eqref{MHD} such that initial data $(u_0, b_0)\in \mathcal{S}(\RR^d)$ satisfies
\[\|u_{0}\|_{H^{\frac{d}{2}-1}}+\|b_{0}\|_{H^{\frac{d}{2}}}\le\delta, \]
and for some $0<t<\delta$,
\[\|b(t,\cdot)\|_{H^{\frac{d}{2}}}>\frac{1}{\delta}. \]
\end{theorem}
\begin{remark}
{Note that the system \eqref{MHD} is {locally} well-posed in $H^{s-1}(\RR^d)\times H^{s}(\RR^d)$ with $s>\frac{d}{2}$, Theorem \ref{H} shows the sharp ill-posedness for \eqref{MHD} in $H^{\frac{d}{2}-1}(\RR^d)\times H^{\frac{d}{2}}(\RR^d)$.
{As it is known,  the Navier-Stokes equations are locally well-posed in the critical Sobolev space $H^{\frac{d}{2}-1}(\RR^d)$. On the other hand, Theorem \ref{H} shows that the non-resistive MHD equations is ill-posed under the framework of the critical Sobolev space $H^{\frac{d}{2}-1}(\RR^d)\times H^{\frac{d}{2}}(\RR^d)$}. Interestingly, in our example,  the ``norm inflation'' happens to the magnetic field not  the flow field {\rm(}see\eqref{flow field1}\rm{)},  which  reflects that  the velocity field
plays a more important role than the magnetic field  in  the interaction between the two fields of the non-resistive MHD system.}
\end{remark}
We are indeed able to prove a stronger statement than Theorem \ref{H}. More precisely, we can show the following main result:
\begin{theorem}[Main result]\label{main}Let $d\ge 2$, $1\le p\le\infty$ and $q>1$. {For any $\delta>0$},  there exists a solution $(u, b)$ to the non-resistive MHD equations \eqref{MHD} such that initial data $(u_0, b_0)\in \mathcal{S}(\RR^d)$ satisfies that the Fourier transforms of $(u_0, b_0)$ are supported on an annulus and
\[\|u_{0}\|_{\dot B^{\frac{d}{p}-1}_{p,q}}+\|b_{0}\|_{ \dot B^{\frac{d}{p}}_{p,q}}\le\delta, \]
and for some $0<t<\delta$,
\[\|b(t,\cdot)\|_{\dot B^{\frac{d}{p}}_{p,q}}>\frac{1}{\delta}. \]
\end{theorem}
Because the Fourier transforms of $(u_0, b_0)$ in Theorem \ref{main} are supported on an annulus, one obtains that $\|u_{0}\|_{ H^{\frac{d}{2}-1}}\approx\|u_{0}\|_{\dot B^{\frac{d}{2}-1}_{2,2}}$ and $\|b_{0}\|_{  H^{\frac{d}{2}}}\approx\|b_{0}\|_{ \dot B^{\frac{d}{2}}_{2,2}}$.
Taking advantage of $H^{\frac{d}{2}}\hookrightarrow \dot H^{\frac{d}{2}} $, we immediately conclude Theorem \ref{H} by Theorem~\ref{main}. Moreover,  one can immediately show the ill-posedness in { the }corresponding nonhomogeneous Besov spaces by the proof of Theorem \ref{main}.
\begin{corollary}The system \eqref{MHD} is ill-posed in $B^{\frac{d}{p}-1}_{p,q}(\RR^d)\times   B^{\frac{d}{p}}_{p,q}(\RR^d)$ for $1\le p\le\infty, q>1$.
\end{corollary}

\begin{remark}
 Recall that the Navier-Stokes equations are {locally} well-posed in $\dot B^{\frac{d}{p}-1}_{p, q}(p<\infty, q<\infty)$ for large initial data and whether  the Navier-Stokes {equations} are well-posed or not for large initial data in $\dot B^{\frac{d}{p}-1}_{p, \infty}(p<\infty)$  remains open. Different from the Navier-Stokes equations, Theorem \ref{main} shows that the non-resistive MHD system \eqref{MHD} is ill-posed in $\dot B^{\frac{d}{p}-1}_{p,q}(\RR^d)\times \dot B^{\frac{d}{p}}_{p,q}(\RR^d)$ with $1\le p\le\infty, q>1$.
\end{remark}


Below we list the local well-posedness/ill-posdeness results of the non-resistive MHD system in the homogeneous Besov spaces $\dot B^{\frac{d}{p}-1}_{p,q}\times\dot B^{\frac{d}{p}}_{p,q}$.
\begin{table}[ht]
\begin{tabular}{p{3cm}|p{4cm}|p{4cm}}
\toprule[1.5pt]
\multicolumn{3}{c} {Local well-posedness/ill-posedness of the system \eqref{MHD} in $\dot B^{\frac{d}{p}-1}_{p,q}\times\dot B^{\frac{d}{p}}_{p,q}$} \\\midrule[1pt]
Results&Range&Category\\\midrule[1pt]
\cite{LTY, YLY}&{$ q=1, \,1\le p\le 2d$}& Local well-posedness\\\hline
\cite{YLY}&{$q=1, \,2d<p<\infty$}& Local existence\\\hline
\cite{B,W,Y}&{$ {q\ge1, \, p=\infty}$}& Ill-posedness\\\hline
Theorem \ref{main}&{$ q>1, \, 1\le p\le\infty$}& Ill-posedness\\\bottomrule[1.5pt]
\end{tabular}
\end{table}

As is shown in the table, our result completes the well-posedness and ill-posedness of the non-resistive MHD equations in critical homogeneous Besov spaces except for the case $2d<p<\infty, q=1$.

\section{Preliminaries}
To begin with,  we review briefly the so-called Littlewood-Paley decomposition theory introduced e.g., in \cite{Ba, C}. Suppose  $(\chi,\varphi)$ be a couple of smooth functions with values in $[0,1]$, where
$\text{supp}\,\chi\subset \big\{\xi\in\RR^d\big{|}|\xi|\le
\frac{4}{3}\big\}$ and $\text{supp}\,\varphi\subset\big\{\xi\in\RR^d\big{|}\frac{3}{4}\le|\xi|\le \frac{8}{3}\big\}$. {Moreover, } {we assume that $\varphi$ satisfies}
\[\sum_{j\in\ZZ}\varphi_j(\xi)=1,\,\, \forall \xi\in \RR^d\backslash \{0\},\qquad\text{where} \, \, \varphi_j(\xi):=\varphi(2^{-j}\xi).\]
Let us define the homogeneous localization operators as follows.
\begin{equation}\nonumber
\begin{split}
&\dot\Delta_j u=\varphi_j(D)u=2^{dj}\int_{\RR^3} {g}(2^jy)u(x-y)\,{\rm d}y,\qquad\forall j\in\ZZ,\\
&\dot S_j u=\chi(2^{-j}D)u=2^{dj}\int_{\RR^3} h(2^jy)u(x-y)\,{\rm d}y,\qquad\forall j\in\ZZ,
\end{split}
\end{equation}
where $g=\mathscr{F}^{-1}\varphi$ and $h=\mathscr{F}^{-1}\chi$. The nonhomogeneous dyadic
blocks $\Delta _j$ are defined by
\begin{align*}
&\Delta_j u=0,\,\,\text{if} \,j\le -2,\qquad \Delta_{-1}u=\chi(D) u=\int_{\RR^3} h(y)u(x-y)\,{\rm d}y,\\
&\Delta_j u=\varphi_j(D)u=2^{dj}\int_{\RR^3} g(2^jy)u(x-y)\,{\rm d}y,\qquad\forall j\ge 0.
\end{align*}
\begin{definition}[\cite {Ba} Homogeneous Besov spaces]\label{def.Be}
Let $s\in\RR$ and $1\le p,q\le\infty,$ The homogeneous Besov space $\dot B^s_{p,q}$ consists of all tempered distributions $u\in \mathcal{S}'_{h}$  such that
\begin{equation}\nonumber
\|u\|_{\dot B^s_{p,q}}\overset{\text{def}}{=}\Big{\|}(2^{js}\|\dot\Delta_j u\|_{L^p})_{j\in\ZZ}\Big{\|}_{\ell^q(\ZZ)}<\infty.
\end{equation}
\end{definition}

\begin{definition}[\cite {Ba} Nonhomogeneous Besov spaces]Let $s\in\RR$ and $1\le p,q\le \infty$. The nonhomogeneous Besov space $B^s_{p,q}$ consists of all tempered distributions $u$ such that
\begin{equation}\nonumber
\|u\|_{B^s_{p,q}}\overset{\text{def}}{=}\Big{\|}(2^{js}\|\Delta_j u\|_{L^p})_{j\in\ZZ}\Big{\|}_{\ell^q(\ZZ)}<\infty.
\end{equation}
\end{definition}
In the context of this paper, we often use the following mixed type time-spatial space.
\begin{definition}Let $T>0$, $s\in\RR$ and $(p, q)\in [1, \infty]^2$. The mixed time-spatial Besov space $\mathcal L^r_T\dot B^s_{p,q}$ consists of all  $u\in\mathcal{S}'_h$ such that
\begin{equation}\nonumber
\|u\|_{\mathcal {L}^r_T\dot B^s_{p,q}}\overset{\text{def}}{=}\Big{\|}(2^{js}\|\dot\Delta_j u\|_{L^r_TL^p})_{j\in\ZZ}\Big{\|}_{\ell^q(\ZZ)}<\infty.
\end{equation}
\end{definition}
\begin{lemma}[\cite{Ba}]\label{trans}
Let $1\le p\le p_1\le\infty$ and $s\in (-d\min\{\frac{1}{p_1}, 1-\frac{1}{p}\}, 1+\frac{d}{p_1}].$ Let $v$ be a vector field such that $\nabla v\in L^1_T(\dot  B^{\frac{d}{p_1}}_{p_1, 1}(\RR^d))$. There exists a constant $C$ depending on $p, s, p_1$ such that all solutions $f\in \mathcal L^\infty_T(\dot B^s_{p,1}(\RR^d))$ of the transport equation
\[\partial_t f+v\cdot\nabla f=g,\,\,\,\,f(0,x)=f_0(x).\]
with initial data $f_0\in \dot  B^s_{p,1}(\RR^d)$ and $g\in L^1_{\text{loc}}(\RR^+;\dot  B^s_{p,1}(\RR^d))$, we have, for $t\in[0, T]$,
\begin{align*}
\|f\|_{\mathcal L^\infty_T(\dot  B^s_{p,1})}\le e^{CV_{p_1}(t)}\Big(\|f_0\|_{ \dot B^s_{p,1}}+\int_0^t  e^{-CV_{p_1}(\tau)}\|g(\tau)\|_{ \dot B^s_{p,1}}\dd\tau\Big),
\end{align*}
where $V_{p_1}(t)=\int_0^t\|\nabla v\|_{\dot B^{\frac{d}{p_1}}_{p_1,1}(\RR^d)}\dd s$.
\end{lemma}
\begin{lemma}[\cite{Dan05}]\label{heat}Let $s\in \RR$ and $1\le r_1, r_2, p, q\le \infty$ with $r_2\le r_1$. Consider the heat equation
\begin{align*}
\partial_t u-\Delta u=f,\qquad
  u(0,x)=u_0(x).
  \end{align*}
Assume that $u_0\in \dot B^s_{p,q}(\RR^d)$ and $f\in \mathcal{L}^{r_2}_T( \dot B^{s-2+\frac{2}{r_2}}_{p,q}(\RR^d))$. Then the above equation has a unique solution $u\in \mathcal{L}^{r_1}_T( \dot B^{s+\frac{2}{r_1}}_{p,q}(\RR^d))$
satisfying
\[\|u\|_{\mathcal{L}^{r_1}_T( \dot B^{s+\frac{2}{r_1}}_{p,q}(\RR^d))}\le C\big(\|u_0\|_{\dot B^s_{p,q}(\RR^d)}+\|f\|_{\mathcal{L}^{r_2}_T( \dot B^{s-2+\frac{2}{r_2}}_{p,q}(\RR^d))}\big).\]
\end{lemma}

\begin{lemma}[\cite{Ba}]\label{Phi}Assume that  $u$ is a smooth vector field. And $\Phi(t,x)$ satisfies
\begin{equation}\label{est-Phi}
\Phi(t,x)=x+\int_0^t u(s, \Phi(s, x))\dd s.
\end{equation}
Then, for all $t\in \RR^{+}$, the flow $\Phi(t,x)$ is a $C^1$ diffeomorphism over $\RR^d$, and we have
\[\|D\Phi^{\pm}(t)\|_{L^\infty_x}\le \exp\big(\int_0^t\|D u(s)\|_{L^\infty_x}\dd s\big).\]
\end{lemma}
{\begin{lemma}[\cite{XZ}]\label{uphi}
Let $u\in \mathcal{S}(\RR^d)$ with $\Div u=0$. The flow $\Phi$ is defined by $u$ in \eqref{est-Phi}.  Then $\Phi$ and the inverse $\Phi^{-}$ are $C^1$ measure-preserving global diffeomorphism over $\RR^d$. There holds that
\begin{align*}
\|u\circ\Phi\|_{\dot B^s_{p,q}}\le C\exp\Big(\int_0^t\|D u(s)\|_{L^\infty_x}\dd s\Big)\|u\|_{\dot B^s_{p,q}}, \quad s\in(-1,1), \quad p, q\in[1,\infty]^2.
\end{align*}
\end{lemma}
\begin{proof}Using Lemma 2.7 in Chapter 2 of \cite{Ba}, we infer that for $1\le p\le\infty$ and any $j,k\in\ZZ$,
\begin{align*}
\|\dot\Delta_j\big((\dot\Delta_k u)\circ\Phi\big)\|_{L^p}\le& C\|\dot\Delta_k u\|_{L^p}\min\Big\{2^{k-j}\|D\Phi^{-}\|_{L^\infty}, 2^{j-k}\|D\Phi\|_{L^\infty}\Big\}\\
\le&C\|\dot\Delta_k u\|_{L^p}(\|D\Phi\|_{L^\infty}+\|D\Phi^-\|_{L^\infty})\min\big\{2^{k-j}, 2^{j-k}\big\}.
\end{align*}
Therefore, for $s\in(-1,1)$ and $u\in\dot B^{s}_{p,q}(\RR^d)$, we have
\begin{align*}
&2^{js}\|\dot\Delta_j(u\circ\Phi)\|_{L^p}\\
\le& 2^{js}\Big(\sum_{k\le j}\|\dot\Delta_j\big((\dot\Delta_k u)\circ\Phi\big)\|_{L^p}+\sum_{k> j}\|\dot\Delta_j\big((\dot\Delta_k u)\circ\Phi\big)\|_{L^p}\Big)\\
\le&C(\|D\Phi\|_{L^\infty}+\|D\Phi^{-}\|_{L^\infty})(\sum_{k\le j}2^{k-j}+\sum_{k> j}2^{j-k})2^{js}\|\dot\Delta_k u\|_{L^p}\\
=&C(\|D\Phi\|_{L^\infty}+\|D\Phi^{-}\|_{L^\infty})(\sum_{k\le j}2^{(k-j)(1-s)}2^{ks}\|\dot\Delta_k u\|_{L^p}+\sum_{k> j}2^{(j-k)(1+s)}2^{ks}\|\dot\Delta_k u\|_{L^p}).
\end{align*}
Then taking $\ell^q$ norm on both sides of the above inequality, thanks to $s\in (-1,1)$, one obtains that
\begin{align*}
\|u\circ\Phi\|_{\dot B^s_{p,q}}\le C(\|D\Phi\|_{L^\infty}+\|D\Phi^{-}\|_{L^\infty})\|u\|_{\dot B^s_{p,q}}.
\end{align*}
Combining with Lemma \ref{Phi}, we complete the proof of this lemma.
\end{proof}}
\section{Proof of Theorem \ref{main}}
{First of all, we introduce the parameters in this section.  Let $N$ be a large enough integer defined later.  For any $q>1$,  $\alpha$ is a constant satisfies that  $\frac{1}{q}<\alpha<1$.}

Before constructing initial data $(u_0, b_0)$, we introduce two smooth functions $\widehat\psi(\xi), \widehat\phi(\xi)$ satisfies that
\begin{small}
\begin{equation}\label{supp psi}
\left\{ \aligned
 &\text{supp}\,{\widehat\psi}(\xi)=\big\{\xi\in\RR^d\big|\,\,1\le \xi_2\le 2, (\ln\ln N)^{-1}\le \xi_i\le 2(\ln\ln N)^{-1}, i\neq 2\big\}:=A,\\
&\widehat\psi(\xi)\equiv1, \,\forall\xi\in\big\{\xi\in\RR^d\big|\,\,\tfrac{5}{4}\le \xi_2\le \tfrac{7}{4}, \tfrac{5}{4}(\ln\ln N)^{-1}\le \xi_i\le \tfrac{7}{4}(\ln\ln N)^{-1}, i\neq 2\big\}:=B,
\endaligned
\right.
\end{equation}
\end{small}
and
\begin{small}
\begin{equation}\label{supp phi}
\left\{ \aligned
 &\text{supp}\,{\widehat\phi}(\xi)=\big\{\xi\in\RR^d\big|\,\,1\le \xi_1\le 2, (\ln\ln N)^{-1}\le \xi_i\le 2(\ln\ln N)^{-1}, i\ge 2\big\}:=C,\\
&\widehat\phi(\xi)\equiv1, \,\forall\xi\in\big\{\xi\in\RR^d\big|\,\,\tfrac{5}{4}\le \xi_1\le \tfrac{7}{4}, \tfrac{5}{4}(\ln\ln N)^{-1}\le \xi_i\le \tfrac{7}{4}(\ln\ln N)^{-1},  i\ge 2 \big\}:=D.
\endaligned
\right.
\end{equation}
\end{small}
We construct initial data $(u_0, b_0)$ as follows:
\begin{equation}\label{initialdata}
\left\{ \aligned
 u_0=&\Big(-\mathscr{F}^{-1}\Big(\frac{\xi_2}{\xi_1}\frac{2^N\widehat{\psi}(2^{-N}\xi)}{2^{Nd}(\ln\ln N)^{3+d}}\Big), \mathscr{F}^{-1}\Big(\frac{2^N\widehat{\psi}(2^{-N}\xi)}{2^{Nd}(\ln\ln N)^{3+d}}\Big), 0, \dots, 0\Big),\\
b_0=&\Big(-\sum_{\frac{N}{2}\le j\le\frac{4N}{5}}\mathscr{F}^{-1}\Big(\frac{\xi_2}{\xi_1}\frac{\widehat{\phi}(2^{-j}\xi)}{2^{jd}j^\alpha}\Big), \sum_{\frac{N}{2}\le j\le\frac{4N}{5}}\mathscr{F}^{-1}\Big(\frac{\widehat{\phi}(2^{-j}\xi)}{2^{jd}j^\alpha}\Big), 0, \dots, 0\Big),
\endaligned
\right.
\end{equation}
where $\mathscr{F}^{-1}$ denotes the inverse Fourier transformation. {Actually, the initial data $(u_0, b_0)$ depends on $N$ from the above definition.}   To begin with, we need to verify that the initial data $(u_0, b_0)$ is small in $\dot B^{\frac{d}{p}-1}_{p,q}(\RR^d)\times\dot  B^{\frac{d}{p}}_{p,q}(\RR^d)$ for all $q>1$.
\subsection{Estimates of initial data $(u_0, b_0)$.}
By the definition of $(u_0, b_0)$ in \eqref{initialdata}, we have
\begin{equation}\label{initaldefine}
\left\{ \aligned
  &\widehat u_0(\xi)=\Big(-\frac{\xi_2}{\xi_1}\frac{2^N\widehat{\psi}(2^{-N}\xi)}{2^{Nd}(\ln\ln N)^{3+d}}, \frac{2^N\widehat{\psi}(2^{-N}\xi)}{2^{Nd}(\ln\ln N)^{3+d}}, 0, \dots, 0\Big),\\
&\widehat b_0(\xi)=\Big(-\frac{\xi_2}{\xi_1}\sum_{\frac{N}{2}\le j\le\frac{4N}{5}}\frac{\widehat{\phi}(2^{-j}\xi)}{2^{jd}j^\alpha}, \sum_{\frac{N}{2}\le j\le\frac{4N}{5}}\frac{\widehat{\phi}(2^{-j}\xi)}{2^{jd}j^\alpha}, 0, \dots, 0\Big).
\endaligned
\right.
\end{equation}
It is easy to verify that $\Div u_0=\Div b_0=0$. Assume that $\widetilde{\psi}(\xi)\in C^\infty_c(\RR^d)$, $\widetilde{\psi}(\xi)\equiv 1$ on $A$ and that {$\widetilde{\psi}$ is supported} in an annulus $\widetilde{A}$, where
\[\widetilde{A}=\Big\{\xi\in\RR^d\big|\,\,\tfrac{1}{2}\le \xi_2\le 3, \tfrac{1}{2}(\ln\ln N)^{-1}\le \xi_i\le 3(\ln\ln N)^{-1}, i\neq 2\Big\}.\]
{Moreover, $\widetilde{\psi}$ satisfies that $\|D^k\widetilde{\psi}\|_{L^{\infty}}\le C(\ln\ln N)^k, \forall k\ge 0.$}

Noting the { $\text{supp}\,\, \widehat\psi(\xi)  \subset\widetilde{A}$, we obtain }
\begin{align*}
&\widehat u_0(\xi)=\frac{2^N}{2^{Nd}(\ln\ln N)^{3+d}}\big(-\tfrac{\xi_2}{\xi_1}{\widetilde{\psi}}(2^{-N}\xi)\widehat{\psi}(2^{-N}\xi), \widehat{\psi}(2^{-N}\xi), 0, \dots, 0\big).
\end{align*}
For $u^{1}_0$, we have $u^{1}_0=\frac{2^N}{(\ln\ln N)^{3+d}}K\ast \psi(2^{N}\cdot)$, where
\begin{align*} K(x)=&-(2\pi)^{-\frac{d}{2}}\int_{\RR^d}\tfrac{\xi_2}{\xi_1}\widetilde{\psi}(2^{-N}\xi)e^{ix\cdot\xi}\dd\xi
=-(2\pi)^{-\frac{d}{2}}2^{Nd}\int_{\RR^d}\tfrac{\xi_2}{\xi_1}\widetilde{\psi}(\xi)e^{i2^Nx\cdot\xi}\dd\xi.
\end{align*}
Let $M=\lfloor 1+\frac{d}{2}\rfloor$\footnote{$\lfloor x\rfloor$ denotes the floor function.}, we have
\begin{align*}
(1+|2^Nx|^2)^M |K(x)|=&-(2\pi)^{-\frac{d}{2}}2^{Nd}\big|\int_{\RR^d}\big(({\rm Id}-\Delta_{\xi})^Me^{i2^Nx\cdot\xi}\big)\tfrac{\xi_2}{\xi_1}\widetilde{\psi}(\xi)\dd\xi\big|\\
=&-(2\pi)^{-\frac{d}{2}}2^{Nd}\big|\int_{\RR^d}\big(({\rm Id}-\Delta_{\xi})^M\tfrac{\xi_2}{\xi_1}\widetilde{\psi}(\xi)\big)e^{i2^Nx\cdot\xi}\dd\xi\big|\\
=&-(2\pi)^{-\frac{d}{2}}2^{Nd}\big|\sum_{|\alpha|+|\beta|\le 2M}c_{\alpha, \beta}\int_{\widetilde{A}}e^{i2^Nx\cdot\xi}\partial^{\alpha}\widetilde{\psi}(\xi)\Big(\partial^{\beta}\frac{\xi_2}{\xi_1}\Big)\dd\xi\big|\\
\le&C2^{Nd}\sum_{|\alpha|+|\beta|\le 2M}c_{\alpha, \beta}(\ln\ln N)^{|\alpha|+|\beta|+1-(d-1)}\\
\le&C2^{Nd}(\ln\ln N)^{2M}.
\end{align*}
Due to $M>\tfrac{d}{2}$, we can infer from the above inequality that
\begin{align*}
\int_{\RR^d}|K(x)|\dd x\le& C(\ln\ln N)^{2M}2^{Nd}\int_{\RR^d}(1+|2^Nx|^2)^{-M}\dd x\\
\le& C(\ln\ln N)^{2M}\int_{\RR^d}(1+|x|^2)^{-M}\dd x\\
\le& C(\ln\ln N)^{2+d}.
\end{align*}
Therefore, with the aid of Young's inequality, one has
\begin{align}
\|u^{1}_0\|_{\dot B^{\frac{d}{p}-1}_{p,q}}&\le \frac{C2^N}{(\ln\ln N)^{3+d}}\|K\|_{L^1}\|\psi(2^N\cdot)\|_{\dot B^{\frac{d}{p}-1}_{p,q}}\label{initial1}\\
&\le\frac{C(\ln\ln N)^{2+d}2^N}{(\ln\ln N)^{3+d}}\|\psi(2^N\cdot)\|_{ B^{\frac{d}{p}-1}_{p,q}}\le \frac{C}{\ln \ln N},\notag\\
\|u^{2}_0\|_{ \dot B^{\frac{d}{p}-1}_{p,q}}&= \frac{2^N}{(\ln\ln N)^{3+d}}\|\psi(2^N\cdot)\|_{ B^{\frac{d}{p}-1}_{p,q}}\le\frac{C}{(\ln\ln N)^{3+d}}\label{initial11}.
\end{align}
In terms of $b_0$, noting the support of $\widehat{\phi}(\xi)$, it is easy to verify that $\frac{\xi_2}{\xi_1}\sim (\ln \ln N)^{-1}$. In the same way as estimating of $\|u_0\|_{B^{\frac{d}{p}-1}_{p,q}}$, we obtain that
\begin{align}\label{initial2}
&\|b_0\|_{\dot B^{\frac{d}{p}}_{p,q}}\sim\Big(\sum_{\frac{N}{2}\le j\le\frac{4N}{5}}2^{\frac{d}{p}jq}\frac{\|\phi(2^j x)\|^q_{L^p}}{j^{\alpha q}}\Big)^{\frac{1}{q}} \sim N^{\frac{1}{q}-\alpha}.
\end{align}
Similarly, it is easy to verify that for $N$ large enough and all $1\le p\le\infty$, we have
\begin{equation}\label{ini}
\left\{ \aligned
  &\|u_0\|_{\dot B^{\frac{d}{p}-1}_{p,1}}\le C(\ln\ln N)^{-1},\,\,\|u_0\|_{ \dot B^{\frac{d}{p}}_{p,1}}\le 2^{N},\,\,\,\,\,\,\,\|u_0\|_{\dot  B^{\frac{d}{p}+1}_{p,1}}\le  2^{2N},\\
&\|b_0\|_{\dot B^{\frac{d}{p}}_{p,1}}\le N^{1-\alpha},\,\,\quad\quad\quad\quad\|b_0\|_{\dot  B^{\frac{d}{p}+1}_{p,1}}\le  2^N, \,\,\,\,\|b_0\|_{ \dot B^{\frac{d}{p}+2}_{p,1}}\le  2^{2N}.
\endaligned
\right.
\end{equation}
\subsection{Local well-posedness {for $(u_0,b_0)$ with the form of  \eqref{initaldefine}}.}
{We begin to establish a {locally} well-posed  result for the system~\eqref{MHD} with the initial data constructed in \eqref{initaldefine}.}
\begin{proposition}\label{localwell}Let initial data $( u_0, b_0)$ be defined by \eqref{initaldefine}. Given $1\le r\le 2d$, there exist  constants $C_0$ and $N_0$ such that for $N>N_0$ and $T=(\ln N)^{-1}2^{-2N}$, the system~\eqref{MHD} has a unique local solution $(u, b)$ associated with initial data $(u_0, b_0)$ satisfying
\begin{align*}
&u\in C([0, T], \B\cap \Bpo)\cap \mathcal L^{1}([0, T], \Bpo\cap \Bppp),\\
&b\in C([0, T], \Bp\cap \Bpp),
\end{align*}
and the following estimates hold for all $t\le T$:
\begin{align}
&\|u\|_{\mathcal  L^\infty_{t}(\B)}+\|u\|_{\mathcal L^1_{t}(\Bpo)}\le 2C_0(\ln\ln N)^{-1},\label{u-1}\\
&\|u\|_{\mathcal  L^\infty_{t}(\Bp)}+\|u\|_{\mathcal L^1_{t}(\Bpp)}\le 2C_02^N,\label{u}\\
&\|u\|_{\mathcal  L^\infty_{t}(\Bpo)}+\|u\|_{\mathcal L^1_{t}(\Bppp)}\le 2C_02^{2N},\label{u1}\\
&\|b\|_{\mathcal  L^\infty_{t}(\Bp)}\le 2C_0N^{1-\alpha}, \label{b}\\
&\|b\|_{\mathcal  L^\infty_{t}(\Bpo)}\le 2C^3_0N^{1-\alpha}2^N,\label{b1}\\
&\|b\|_{\mathcal  L^\infty_{t}(\Bpp)}\le 2C^4_0N^{1-\alpha}2^{2N}. \label{b2}
\end{align}
\end{proposition}
\begin{proof}{ According to the local well-posedness theory of system (1.1) in [18], there exists a positive time $\bar{T}_0$ such that the system  (1.1) possesses a unique solution $(u, b)\in \dot{B}^{\frac{d}{r}-1}_{r,1} \times \dot{B}^{\frac{d}{r} }_{r,1} $ with initial data $(u_0, b_0)$ satisfying
\begin{align*}
&u\in C([0, \bar{T}_0], \dot{B}^{\frac{d}{r}-1}_{r,1} )\cap \mathcal L^{1}([0, \bar{T}_0], \dot{B}^{\frac{d}{r}+1}_{r,1} ),\\
&b\in C([0, \bar{T}_0], \dot{B}^{\frac{d}{r} }_{r,1} ).
\end{align*}
Furthermore, the uniform estimates in [18] shows that, for any small enough $\eta$, there exist $C_0>1$ and a $0<T_{\eta}\le \bar{T}_0$, where $T_{\eta}$ depends on $u_0$ and $b_0$, such that for any $0<T_0\le T_{\eta}$,
\begin{equation}\label{well-ub}
\left\{ \aligned
   &\|u(t)\|_{ \mathcal L^\infty([0,T_0]; \dot{B}^{\frac{d}{r}-1}_{r,1} )}+\|b(t)\|_{\mathcal L^\infty([0,T_0]; \dot{B}^{\frac{d}{r} }_{r,1} )}\leq C_0(\|u_0\|_{   \dot{B}^{\frac{d}{r}-1}_{r,1}  }+\|b_0\|_{  \dot{B}^{\frac{d}{r}  }_{r,1} }),\\
&\|u\|_{ L^2([0, {T}_0], \dot{B}^{\frac{d}{r}}_{r,1} )\cap \mathcal L^{1}([0, {T}_0], \dot{B}^{\frac{d}{r}+1}_{r,1} )}\leq \eta.
\endaligned
\right.
\end{equation}
Since $(u_0, b_0)\in \mathcal{S}(\RR^d)$, one can  deduce that for short time $T_0$,
\begin{align}
&u\in C([0, T_0], \dot{B}^{\frac{d}{r}-1}_{r,1}\cap \dot{B}^{\frac{d}{r}+1}_{r,1})\cap \mathcal L^{1}([0, T_0], \dot{B}^{\frac{d}{r}+1}_{r,1}\cap \dot{B}^{\frac{d}{r}+3}_{r,1}),\notag\\
&b\in C([0, T_0], \dot{B}^{\frac{d}{r} }_{r,1}\cap \dot{B}^{\frac{d}{r}+2}_{r,1}).\notag
\end{align}
Indeed, using Lemma 2.1--Lemma 2.2, one has that
\begin{align}
&\|u\|_{\mathcal L^\infty([0,T_0]; \dot{B}^{\frac{d}{r}+1}_{r,1})\cap \mathcal L^{1}([0, T_0],   \dot{B}^{\frac{d}{r}+3}_{r,1})}  \notag\\
\leq& \|u_0\|_{\dot{B}^{\frac{d}{r}+1}_{r,1}}
+
\int_{0}^{T_0}\|(u\cdot\nabla u,~b\cdot\nabla b)\|_{ \dot{B}^{\frac{d}{r}+1}_{r,1} }
\dd t, \notag\\
\leq & \|u_0\|_{\dot{B}^{\frac{d}{r}+1}_{r,1}}
+C\|u\|_{L^2([0, T_0], \dot{B}^{\frac{d}{r}}_{r,1})}\|u\|_{L^2([0, T_0], \dot{B}^{\frac{d}{r}+2}_{r,1})}
+C\int_{0}^{T_0}\|b(t)\|_{ \dot{B}^{\frac{d}{r}}_{r,1}}\|b(t)\|_{\dot{B}^{\frac{d}{r}+2}_{r,1}}\dd t,\notag\\
\leq & \|u_0\|_{\dot{B}^{\frac{d}{r}+1}_{r,1}}
+C\eta\|u\|_{\mathcal L^\infty([0,T_0]; \dot{B}^{\frac{d}{r}+1}_{r,1})\cap \mathcal L^{1}([0, T_0],   \dot{B}^{\frac{d}{r}+3}_{r,1})}\notag\\
&+C(\|u_0\|_{   \dot{B}^{\frac{d}{r}-1}_{r,1}  }+\|b_0\|_{  \dot{B}^{\frac{d}{r}  }_{r,1} })T_0\|b\|_{\mathcal L^\infty([0,T_0];\dot{B}^{\frac{d}{r}+2}_{r,1})}.\label{est-u}
\end{align}
Taking $\nabla$ on E.q.$(1.1)_2$ and using Lemma 2.2, we obtain that
\begin{align*}
&\|b\|_{\mathcal L^\infty([0,T_0];\dot{B}^{\frac{d}{r}+2}_{r,1})}
\leq \|b_0\|_{\dot{B}^{\frac{d}{r}+2}_{r,1}} \\
&+C\int_{0}^{T_0}\|u(t)\|
_{ \dot{B}^{\frac{d}{r}+1}_{r,1}}\|b(t)\|_{\dot{B}^{\frac{d}{r}+2}_{r,1}}
+\|u(t)\|_{ \dot{B}^{\frac{d}{r}+2}_{r,1}}\|b(t)\|_{\dot{B}^{\frac{d}{r} +1}_{r,1}}+\|u(t)\|_{ \dot{B}^{\frac{d}{r}+3}_{r,1}}\|b(t)\|_{\dot{B}^{\frac{d}{r} }_{r,1}}\dd t  \\
\leq&\|b_0\|_{\dot{B}^{\frac{d}{r}+2}_{r,1}}+C\eta\|b(t)\|_{\mathcal L^\infty([0,T_0];\dot{B}^{\frac{d}{r}+2}_{r,1})}\\
&+C\sqrt{T_0}\|b\|_{\mathcal L^\infty([0,T_0];\dot{B}^{\frac{d}{r}+2}_{r,1})}\|u\|_{\mathcal L^\infty([0,T_0]; \dot{B}^{\frac{d}{r}+1}_{r,1})\cap \mathcal L^{1}([0, T_0],   \dot{B}^{\frac{d}{r}+3}_{r,1})}\\
&+C\sqrt{T_0}(\|u_0\|_{   \dot{B}^{\frac{d}{r}-1}_{r,1}  }+\|b_0\|_{  \dot{B}^{\frac{d}{r}  }_{r,1} })\|u\|_{\mathcal L^\infty([0,T_0]; \dot{B}^{\frac{d}{r}+1}_{r,1})\cap \mathcal L^{1}([0, T_0],   \dot{B}^{\frac{d}{r}+3}_{r,1})}\\
&+C(\|u_0\|_{   \dot{B}^{\frac{d}{r}-1}_{r,1}  }+\|b_0\|_{  \dot{B}^{\frac{d}{r}  }_{r,1} })\|u(t)\|_{\mathcal L^{1}([0, T_0],   \dot{B}^{\frac{d}{r}+3}_{r,1})}.
\end{align*}
Plugging estimate \eqref{est-u} into the above inequality yields that
\begin{align}
&\|b\|_{\mathcal L^\infty([0,T_0];\dot{B}^{\frac{d}{r}+2}_{r,1})}
\leq\|b_0\|_{\dot{B}^{\frac{d}{r}+2}_{r,1}}+C(\|u_0\|_{   \dot{B}^{\frac{d}{r}-1}_{r,1}  }+\|b_0\|_{  \dot{B}^{\frac{d}{r}  }_{r,1} })\|u_0\|_{\dot{B}^{\frac{d}{r}+1}_{r,1}}\notag\\
&+C(\eta+(\|u_0\|_{   \dot{B}^{\frac{d}{r}-1}_{r,1}  }+\|b_0\|_{  \dot{B}^{\frac{d}{r}  }_{r,1} })T_0)\|b(t)\|_{\mathcal L^\infty([0,T_0];\dot{B}^{\frac{d}{r}+2}_{r,1})}\notag\\
&+C\sqrt{T_0}\|b\|_{\mathcal L^\infty([0,T_0];\dot{B}^{\frac{d}{r}+2}_{r,1})}\|u\|_{\mathcal L^\infty([0,T_0]; \dot{B}^{\frac{d}{r}+1}_{r,1})\cap \mathcal L^{1}([0, T_0],   \dot{B}^{\frac{d}{r}+3}_{r,1})}\notag\\
&+C(\sqrt{T_0}+\eta)(\|u_0\|_{   \dot{B}^{\frac{d}{r}-1}_{r,1}  }+\|b_0\|_{  \dot{B}^{\frac{d}{r}  }_{r,1} })\|u\|_{\mathcal L^\infty([0,T_0]; \dot{B}^{\frac{d}{r}+1}_{r,1})\cap \mathcal L^{1}([0, T_0],   \dot{B}^{\frac{d}{r}+3}_{r,1})}.\label{est-b}
\end{align}
Collecting the estimates \eqref{est-u} and \eqref{est-b} together, for taking $\eta$ and $T_0$ small enough, by continuity argument, we have
\begin{align*}
&\|u\|_{\mathcal L^\infty([0,T_0]; \dot{B}^{\frac{d}{r}+1}_{r,1})\cap \mathcal L^{1}([0, T_0],   \dot{B}^{\frac{d}{r}+3}_{r,1})} +\|b\|_{\mathcal L^\infty([0,T_0]; \dot{B}^{\frac{d}{r}+2}_{r,1})}\\
\leq& C(\|u_0\|_{\dot{B}^{\frac{d}{r}-1}_{r,1}}, \|b_0\|_{\dot{B}^{\frac{d}{r}}_{r,1}})(\|u_0\|_{\dot{B}^{\frac{d}{r}+1}_{r,1}}+\|b_0\|_{\dot{B}^{\frac{d}{r}+2}_{r,1}}
).
\end{align*}}
Now we need to show that $T_0$ can be extended to $(\ln N)^{-1}2^{-2N}$. To prove this, we are focused on showing a priori estimates \eqref{u-1}-\eqref{b2} on $[0, 2(\ln N)^{-1}2^{-2N})$.

With the aid of Lemma \ref{trans},  we have
\begin{equation}\label{bBp}
\|b\|_{\mathcal L^\infty_T(\Bp)}\le\exp(C\|u\|_{L^1_T( B^{\frac{d}{r}+1}_{r,1})})(\|b_0\|_{\Bp}+\|b\|_{L^\infty_T(\Bp)}\|u\|_{ L^1_T (B^{\frac{d}{r}+1}_{r,1})}).
\end{equation}
Taking advantage of {Lemma} \ref{heat}, one can deduce that
\begin{equation}\label{uB}
\|u\|_{\mathcal L^\infty_T(\B)}+\|u\|_{\mathcal L^1_T(\Bpo)}
\le C(\|u_0\|_{\B}+\|u\|^2_{\mathcal L^2_T(\Bp)}+T\|b\|^2_{\mathcal L^\infty_T(\Bp)}).
\end{equation}
Similarly, we can obtain that
\begin{align}
\|b\|_{\mathcal L^\infty_T(\Bpo)}
\le&\exp(C\|u\|_{\mathcal L^1_T( B^{\frac{d}{r}+1}_{r,1})})(\|b_0\|_{\Bpo}\notag\\
&+\|b\|_{\mathcal L^\infty_T(\Bp)}\|u\|_{\mathcal L^1_T (B^{\frac{d}{r}+2}_{r,1})}+\|b\|_{\mathcal L^\infty_T\Bpo}\|u\|_{ \mathcal L^1_T(B^{\frac{d}{r}+1}_{r,1})}).\label{bBpo}
\end{align}
With the aid of {Lemma} \ref{heat}, one can infer that
\begin{align}
&\|u\|_{\mathcal L^\infty_T(\Bp)}+\|u\|_{\mathcal L^1_T(\Bpp)}\notag\\
\le&C(\|u_0\|_{\Bp}+\sqrt{T}\|u\|_{\mathcal L^\infty_T(\Bp)}\|u\|_{\mathcal L^2_T(\Bpo)}+T\|b\|_{\mathcal L^\infty_T(\Bp)}\|b\|_{\mathcal L^\infty_T(\Bpo)}),\label{uBp}
\end{align}
and
\begin{align}
&\|u\|_{\mathcal L^\infty_T(\Bpo)}+\|u\|_{\mathcal L^1_T(\Bppp)}\notag\\
\le&C(\|u_0\|_{\Bpo}+\sqrt{T}\|u\|_{\mathcal L^\infty_T(\Bp)}\|u\|_{\mathcal L^2_T(\Bpp)}\notag\\
&+T\|b\|_{\mathcal L^\infty_T(\Bp)}\|b\|_{\mathcal L^\infty_T(\Bpp)}+T\|b\|^2_{\mathcal L^\infty_T(\Bpo)}).\label{uBpo}
\end{align}
Taking derivative on E.q.$\eqref{MHD}_2$, we obtain by Lemma \ref{trans} that
\begin{align}
\| b\|_{\mathcal L^\infty_T(\Bpp)}
\le& \exp (C\|u\|_{\mathcal L^1_T(\dot B^{\frac{d}{r}+1}_{r,1})})(\| b_0\|_{\Bpp}+C\|b\|_{\mathcal L^\infty_T\Bp}\|u\|_{\mathcal L^1_T\dot B^{\frac{d}{r}+3}_{r,1}}\notag\\
&+CT\|u\|_{\mathcal L^\infty_T \Bpo}\|b\|_{\mathcal L^\infty_T \Bpp}+C\sqrt{T}\|b\|_{\mathcal L^\infty_T  \Bpo}\|u\|_{\mathcal L^2_T \Bpp}).\label{bBpp}
\end{align}
Since initial data $(u_0, b_0)$ satisfies \eqref{ini}, let constant $C_0>\max\{2C^2, 16\}$, we can infer from the above estimates that there exist a positive time $T_1$ such that for $t\le T_1$, estimates \eqref{u-1}-\eqref{b2} hold. We define
$$T^*:=\sup\big\{0<t\le 2^{-2N+1}(\ln N)^{-1}\big|\eqref{u-1}-\eqref{b2}~~~ {\text{hold on the time interval} }~~~ [0, t]\big\}.$$
If $T^*=2^{-2N+1}(\ln N)^{-1}$, we complete the proof. Otherwise, for $t<T^*<2^{-2N+1}(\ln N)^{-1}$,
combining \eqref{u-1} with \eqref{bBp}, we have
\begin{align*}
\|b\|_{\mathcal L^\infty_T(\Bp)}\le&\exp(2CC_0(\ln\ln N)^{-1})(N^{1-\alpha}+4C^2_0N^{1-\alpha}(\ln\ln N)^{-1})
\le C_0N^{1-\alpha},
\end{align*}
where the last inequality holds for large enough $N$. Utilizing \eqref{uB} and the above inequality, one has
\begin{align*}
\|u\|_{\mathcal L^\infty_T(\B)}+\|u\|_{\mathcal L^1_T(\Bpo)}
\le& C(C(\ln\ln N)^{-1}+4C_0(\ln\ln N)^{-2}+2^{-4N}(\ln N)^{-2}C^2_0N^{2-2\alpha})\\
\le&C_0(\ln\ln N)^{-1}.
\end{align*}
In the same way as deriving the above inequality, plugging \eqref{u-1}, \eqref{b} \eqref{b1} and \eqref{u} into \eqref{bBpo} yields that for large $N$,
\begin{align*}
\|b\|_{\mathcal L^\infty_T(\Bpo)}
\le&\exp(2CC_0(\ln\ln N)^{-1})(2^N+4C^2_0N^{1-\alpha}2^N+4C^4_0N^{1-\alpha}2^N(\ln \ln N)^{-1})\\
\le& C^3_0N^{1-\alpha}2^N.
\end{align*}
Similarly, owing to \eqref{u}, \eqref{b} and \eqref{b1}, we have
\begin{equation}\nonumber
\begin{aligned}
\|u\|_{\mathcal L^\infty_T(\Bp)}+\|u\|_{\mathcal L^1_T(\Bpp)}
\le&C(2^N+8C^2_0(\ln N)^{-\frac{1}{2}}2^{N}+8C^4_0(\ln N)^{-1}N^{2-2\alpha}2^{-N})
\le C_02^N,
\end{aligned}
\end{equation}
and we can deduce by  \eqref{u}-\eqref{b2} that
\begin{equation}\nonumber
\begin{aligned}
\|u\|_{\mathcal L^\infty_T(\Bpo)}+\|u\|_{\mathcal L^1_T(\Bppp)}
\le&C(2^{2N}+8C^2_0(\ln N)^{-\frac{1}{2}}2^{2N}+8C^2_0(\ln N)^{-1}2^{2N}\\
&+8C^5_0N^{2-2\alpha}(\ln N)^{-1}N^{1-\alpha}+8C^6_0(\ln N)^{-1}N^{2-2\alpha})\\
\le&C_0 2^{2N}.
\end{aligned}
\end{equation}
Finally, plugging the above inequality, \eqref{u-1}, \eqref{b}-\eqref{b2} into \eqref{bBpp}, one obtains that
\begin{align*}
\| b\|_{\mathcal L^\infty_T(\Bpp)}
\le& \exp (2CC_0(\ln \ln N)^{-1})(2^{2N}+4CC^2_0N^{1-\alpha}2^{2N}+16CC^5_0(\ln N)^{-1}N^{1-\alpha}2^{2N})\\
\le&C^4_0N^{1-\alpha}2^{2N}.
\end{align*}
The above estimates contradict to the definition of $T^{\ast}$. Therefore, $T^{\ast}=2^{-2N+1}(\ln N)^{-1}$.
\end{proof}

\subsection{ Norm inflation.}
{In this section, we  show a norm inflation  phenomenon for the magnetic field $b(t,x)$. Before doing this, we define the flow map} $\Phi(x,t)$ by
\begin{equation}\label{def-phi}
\left\{ \aligned
    & \frac{\dd \Phi(x,t)}{\dd t}=u(t, \Phi(x,t)),\\
     &\Phi(x,t)|_{t=0}=x,
\endaligned
\right.
\end{equation}
{In the following, $T=(\ln N)^{-1}2^{-2N}$.} We rewrite the magnetic field $b(t,x)$ on $[0, T]$ as follows:
\begin{align*}
b(t, \Phi(x,t))=b_0(x)+\int_0^t(b\cdot\nabla u)(s, \Phi(x, s))\dd s.
\end{align*}
Based on the above equality,  we decompose $b(T, \Phi(x,T))$ into the following three parts:
\begin{align*}
b(T, \Phi(x,T))
=&b_0(x)+\underbrace{\int_0^T(b_0\cdot\nabla e^{t\Delta}u_0)(t,x)\dd t}_{I^{\rm B}}\\
&+\underbrace{\int_0^T(b\cdot\nabla u)(s, \Phi(x, s))\dd s-\int_0^T(b_0\cdot\nabla e^{t\Delta}u_0)(t,x)\dd t}_{I^{\rm S}}.
\end{align*}
Now, we aim to estimate the lower bound of $\|I^{\rm B}\|_{\dot B^0_{\infty, q}}$ and the upper bound of $\|I^{\rm S}\|_{\dot B^0_{\infty, q}}$.
\textbf{Estimates of $\|I^{\rm B}\|_{\dot B^0_{\infty, q}}$.}
Due to $\dot B^0_{\infty, q}\hookrightarrow \dot B^0_{\infty, \infty}$ and $\|f\|_{L^\infty}\ge |f(0)|=|\int_{\RR^d}\hat{f}(\xi)\dd \xi|$, we have
\begin{align*}
\big\|I^{\rm B}\big\|_{\dot B^0_{\infty,q}}\ge&C \Big\|\dot\Delta_N\int_0^{T}b_0\cdot\nabla e^{s\Delta}u^{2}_0\dd s\Big\|_{L^\infty}\\
\ge&C\Big|\int_{\RR^d}\int_0^T\int_{\RR^d}\widehat\varphi(2^{-N}\xi)\widehat b_0(\eta)\cdot (\xi-\eta)e^{-s|\xi-\eta|^2}\widehat{u}^{2}_0(\xi-\eta)\dd\eta\dd s\dd\xi\Big|\\
\ge&C\Big|\int_{\RR^d}\int_0^T\int_{\RR^d}\widehat\varphi(2^{-N}\xi)\widehat b^{2}_0(\eta) (\xi_2-\eta_2)e^{-s|\xi-\eta|^2}\widehat{u}^{2}_0(\xi-\eta)\dd\eta\dd s\dd\xi\Big|\\
&-C\Big|\int_{\RR^d}\int_0^T\int_{\RR^d}\widehat\varphi(2^{-N}\xi)\widehat b^{1}_0(\eta) (\xi_1-\eta_1)e^{-s|\xi-\eta|^2}\widehat{u}^{2}_0(\xi-\eta)\dd\eta\dd s\dd\xi\Big|\\
:=&I^{\rm B}_1-I^{\rm B}_2.
\end{align*}
For $I^{\rm B}_1$, by the definitions of $b_0$ and $u_0$, taking change of variables, one yields that
\begin{align*}
I^{\rm B}_1
=&C\Big|\int_{\RR^d}\int_{\RR^d}\widehat\varphi(2^{-N}\xi)\widehat b^{2}_0(\eta) (\xi_2-\eta_2)\frac{1-e^{-T|\xi-\eta|^2}}{|\xi-\eta|^2}\widehat{u}^{2}_0(\xi-\eta)\dd\eta\dd\xi\Big|\\
=&C\Big|\int_{\RR^d}\int_{\RR^d}\widehat\varphi(2^{-N}\xi)\sum_{\frac{N}{2}\le j\le\frac{4N}{5}}\frac{\widehat\phi(2^{-j}\eta)}{2^{jd}j^\alpha} (\xi_2-\eta_2)\frac{1-e^{-T|\xi-\eta|^2}}{|\xi-\eta|^2}\frac{2^N\widehat\psi(2^{-N}(\xi-\eta))}{2^{Nd}(\ln\ln N)^{3+d}}\dd\eta\dd\xi\Big|\\
=&C\Big|\int_{\RR^d}\int_{\RR^d}\widehat\varphi(\tilde\xi)\sum_{\frac{N}{2}\le j\le\frac{4N}{5}}\frac{\widehat\phi(\tilde\eta)}{j^\alpha} (\tilde\xi_2-2^{j-N}\tilde\eta_2)\frac{1-e^{-T2^{2N}|\tilde\xi-2^{j-N}\tilde\eta|^2}}{|\tilde\xi-2^{j-N}\tilde\eta|^2}\frac{\widehat\psi(\tilde\xi-2^{j-N}\tilde\eta)}{(\ln\ln N)^{3+d}}
\dd\tilde\eta\dd\tilde\xi\Big|.
\end{align*}
{Noting the fact that $\widehat\varphi$, $\widehat\phi$ and $\widehat\psi$ support on an annulus {$\mathcal{C}:=\{\xi\in\RR^d| \tfrac{3}{4}\le|\xi|\le\tfrac{8}{3}\}$}, and $|2^{j-N}\tilde\eta|\le 2^{-\frac{N}{5}}|\tilde\eta|\ll 1$ for large enough $N$, we obtain that there exist constants $c_0$, $c_1$ such that
\begin{equation}\label{1}
c_0\le{ |\tilde\xi-2^{j-N}\tilde\eta|^2}\le c_1,\quad \tilde\xi_2-2^{j-N}\tilde\eta_2\sim 1, \quad{\rm if}\,\tilde\xi\in {\rm supp}\,\widehat\varphi(\tilde\xi), \,\tilde\eta\in {\rm supp}\, \widehat\phi(\tilde\eta).
\end{equation}
Therefore, we can easily deduce from $T=2^{-2N}(\ln N)^{-1}$ that
\begin{align*}
I^{\rm B}_1
\ge C_0(1-e^{-c_0(\ln N)^{-1}})(\ln\ln N)^{2-2d}(\ln\ln N)^{-3-d}N^{1-\alpha}.
\end{align*}
Owning to
\begin{equation}\nonumber
\frac{\tilde\eta_2}{\tilde\eta_1}\sim (\ln\ln N)^{-1}, \,\,\,{\rm if}\,\tilde\eta\in {\rm supp}\, \widehat\phi(\tilde\eta); \quad\tilde\xi_1\sim (\ln\ln N)^{-1},\,\,\,{\rm if}\,\tilde\xi\in {\rm supp}\, \widehat\psi(\tilde\xi),
\end{equation}
combining with \eqref{1}, we have
\begin{align*}
I^{\rm B}_2=&C\Big|\int_{\RR^d}\int_{\RR^d}\widehat\varphi(2^{-N}\xi)\sum_{\frac{N}{2}\le j\le\frac{4N}{5}}\frac{\eta_2}{\eta_1}\frac{\widehat\phi(2^{-j}\eta)}{2^{jd}j^\alpha} (\xi_1-\eta_1)\frac{1-e^{-T|\xi-\eta|^2}}{|\xi-\eta|^2}\frac{2^N\widehat\psi(2^{-N}(\xi-\eta))}{2^{Nd}{(\ln \ln N)^{3+d}}}\dd\eta\dd\xi\Big|\\
=&C\Big|\int_{\RR^d}\int_{\RR^d}\widehat\varphi(\tilde\xi)\sum_{\frac{N}{2}\le j\le\frac{4N}{5}}\frac{\tilde\eta_2}{\tilde\eta_1}\frac{\widehat\phi(\tilde\eta)}{j^\alpha} (\tilde\xi_1-2^{j-N}\tilde\eta_1)\frac{1-e^{-T2^{2N}|\tilde\xi-2^{j-N}\tilde\eta|^2}}{|\tilde\xi-2^{j-N}\tilde\eta|^2}\frac{\widehat\psi(\tilde\xi-2^{j-N}\tilde\eta)}{{(\ln \ln N)^{3+d}}}
\dd\tilde\eta\dd\tilde\xi\Big|\\
\le&C_1(1-e^{-c_1(\ln N)^{-1}})(\ln\ln N)^{-2d}(\ln\ln N)^{-3-d}N^{1-\alpha}.
\end{align*}
Hence, we can deduce from the above two estimates that
\begin{align}
\big\|I^{\rm B}\big\|_{\dot B^{0}_{\infty,q}}
\ge& C_0(1-{e^{-c_0(\ln N)^{-1}}})(\ln\ln N)^{2-2d}(\ln\ln N)^{-3-d}N^{1-\alpha}\notag\\
&-C_1(1-e^{-c_1(\ln N)^{-1}})(\ln\ln N)^{-2d}(\ln\ln N)^{-3-d}N^{1-\alpha}.\label{Ib}
\end{align}
Choosing $N$ large enough such that $c_0(\ln N)^{-1}\le\frac{1}{2}$, utilizing the following inequality
\[1-e^{-x}\ge\frac{x}{2} \,\,\text{for}\,\, x\in [0, \tfrac{1}{2}],\qquad 1-e^{-x}<x \,\,\text{for}\, \,x>0,\]
one can infer  from \eqref{Ib} that
\begin{align}
\big\|I^{\rm B}\big\|_{\dot B^{0}_{\infty,q}}
\ge&\frac{C_0c_0}{2}(\ln N)^{-1}(\ln\ln N)^{2-2d}(\ln\ln N)^{-3-d}N^{1-\alpha}\notag\\
&-C_1c_1(\ln N)^{-1}(\ln\ln N)^{-2d}(\ln\ln N)^{-3-d}N^{1-\alpha}\notag\\
\ge&C(\ln N)^{-1}(\ln\ln N)^{-1-3d}N^{1-\alpha}.\label{IB}
\end{align}

\textbf{Estimates of $\|I^{\rm S}\|_{\dot B^{0}_{\infty,q}}$.}
We decompose $\|I^{\rm S}\|_{\dot B^{0}_{\infty,q}}$ into the following three parts:
\begin{align*}
\|I^{\rm S}\|_{\dot B^{0}_{\infty,q}}
\le&\big\|\int_0^T(b\circ \Phi-b_0)\cdot ((\nabla u)\circ \Phi)\dd t\big\|_{\dot B^{0}_{\infty,q}}\\
&+\big\|\int_0^Tb_0\cdot ((\nabla u)\circ \Phi-(\nabla u)(x))\dd t\big\|_{\dot B^{0}_{\infty,q}}\\
&+\big\|\int_0^Tb_0\cdot \nabla (u-e^{t\Delta}u_0)(x)\dd t\big\|_{\dot B^{0}_{\infty,q}}\\
:=&I^{\rm S}_1+I^{\rm S}_2+I^{\rm S}_3.
\end{align*}
For $I^{\rm S}_1$, using $\dot B^{\frac{d}{p}}_{p, 1}\hookrightarrow \dot B^{0}_{\infty, q}$ and Lemma \ref{uphi}, we obtain for $d<p\le 2d$ that
\begin{align}
I^{\rm S}_1\le& CT\|b\circ \Phi-b_0\|_{L^\infty_T\dot B^{\frac{d}{p}}_{p,1}}\|(\nabla u)\circ \Phi\|_{L^\infty_T\dot B^{\frac{d}{p}}_{p,1}}\notag\\
\le& C{\exp\Big(\int_0^T\|D u(s)\|_{L^\infty_x}\dd s\Big)}T\|b\cdot\nabla u\|_{L^1_T\dot B^{\frac{d}{p}}_{p,1}}\|u \|_{L^\infty_T\dot B^{\frac{d}{p}+1}_{p,1}},\label{IS1}
\end{align}
With the aid of Proposition \ref{localwell}, we have for $T\le T_0$,
\begin{align}
\exp\Big(\int_0^T\|D u(s)\|_{L^\infty_x}\dd s\Big)\le&\exp\big( T\|u\|_{\mathcal L^{\infty}([0, T],\dot B^{\frac{d}{p}+1}_{p,1})}\big)\notag\\
\le& \exp(2^{-2N}(\ln N)^{-1}\cdot 2C_02^{2N})\le C.\label{Psi}
\end{align}
Combining the above inequality with \eqref{IS1}, \eqref{u1} and \eqref{b}, one obtains that
\begin{equation}\label{Is1}
\begin{aligned}
I^{\rm S}_1
\le CT^2\|b\|_{L^\infty_T\dot B^{\frac{d}{p}}_{p,1}}\|u\|^2_{L^\infty_T\dot B^{\frac{d}{p}+1}_{p,1}}
\le C2^{-4N}(\ln N)^{-2}N^{1-\alpha}2^{4N}
\le C(\ln N)^{-2}N^{1-\alpha}.
\end{aligned}
\end{equation}
In terms of $I^S_2$, taking advantage of Newton Leibniz formula, one can deduce that
\begin{align*}
I^S_2\le& C\sqrt{T}\|b_0\|_{L^\infty_T\dot B^{\frac{d}{p}}_{p,1}}\|(\nabla u)\circ \Phi-(\nabla u)(x)\|_{L^2_T\dot B^{\frac{d}{p}}_{p,1}}\\
=& C\sqrt{T}\|b_0\|_{L^\infty_T\dot B^{\frac{d}{p}}_{p,1}}\|\int_0^1\partial_{\theta}((\nabla u)(\theta\Phi+(1-\theta) x))\dd \theta\|_{L^2_T\dot B^{\frac{d}{p}}_{p,1}}\\
=&C\sqrt{T}\|b_0\|_{L^\infty_T\dot B^{\frac{d}{p}}_{p,1}}\|\int_0^1(D^2 u)(\theta\Phi+(1-\theta) x)\cdot (\Phi-x)\dd \theta\|_{L^2_T\dot B^{\frac{d}{p}}_{p,1}}\\
\le&C\sqrt{T}\|b_0\|_{L^\infty_T\dot B^{\frac{d}{p}}_{p,1}}\|\Phi-x\|_{L^\infty_T\dot B^{\frac{d}{p}}_{p,1}}\int_0^1\|(D^2 u)(\theta\Phi+(1-\theta) x)\|_{L^2_T\dot B^{\frac{d}{p}}_{p,1}}\dd \theta.
\end{align*}
{With the aid of \eqref{def-phi}, Lemma \ref{uphi} and \eqref{Psi}, we have
\begin{align}
I^S_2\le& C\sqrt{T}\|b_0\|_{L^\infty_T\dot B^{\frac{d}{p}}_{p,1}}\Big\|\int_0^t u\circ\Phi\dd s\Big\|_{L^\infty_T\dot B^{\frac{d}{p}}_{p,1}}\|u\|_{L^2_T\dot B^{\frac{d}{p}+2}_{p,1}}\notag\\
\le&C\sqrt{T}T\|b_0\|_{L^\infty_T\dot B^{\frac{d}{p}}_{p,1}}\|u\|_{L^\infty_T\dot B^{\frac{d}{p}}_{p,1}}\|u\|_{L^2_T\dot B^{\frac{d}{p}+2}_{p,1}}.\label{Is2-1}
\end{align}
By \eqref{u1} in Proposition \ref{localwell}, we have that for $T\le 2^{-2N}(\ln N)^{-1}$ and $d<p\le 2d$,
\begin{align*}
\|u\|_{L^2_T\dot B^{\frac{d}{p}+2}_{p,1}}\le \|u\|_{\mathcal  L^\infty_{t}\dot B^{\frac{d}{p}+1}_{p,1}}+\|u\|_{\mathcal L^1_{t}\dot B^{\frac{d}{p}+3}_{p,1}}\le 2C_02^{2N}.
\end{align*}
Plugging the above estimate, \eqref{ini} and \eqref{u} into  \eqref{Is2-1} yields that
\begin{equation}
\begin{aligned}
I^S_2\le& C2^{-3N}(\ln N)^{-\frac{3}{2}}\cdot 2^{N}\cdot 2^{2N}\le C N^{1-\alpha} (\ln N)^{-\frac{3}{2}}.\label{Is2}
\end{aligned}
\end{equation}}
By Lemma \ref{heat}, we can bound $I^S_3$ as follows:
\begin{align*}
I^S_3\le& CT\|b_0\|_{L^\infty_T\dot B^{\frac{d}{p}}_{p,1}}\|u-e^{t\Delta}u_0\|_{L^\infty_T\dot B^{\frac{d}{p}+1}_{p,1}}\\
\le &C T\|b_0\|_{L^\infty_T\dot B^{\frac{d}{p}}_{p,1}}(\|u\cdot\nabla u\|_{L^1_T\dot B^{\frac{d}{p}+1}_{p,1}}+\|b\cdot\nabla b\|_{L^1_T\dot B^{\frac{d}{p}+1}_{p,1}})\\
\le&CT\|b_0\|_{L^\infty_T\dot B^{\frac{d}{p}}_{p,1}}(\sqrt{T}\|u\|_{L^\infty_T\dot B^{\frac{d}{p}}_{p,1}}\|u\|_{L^2_T\dot B^{\frac{d}{p}+2}_{p,1}}\\
&+T\|b\|_{L^\infty_T\dot B^{\frac{d}{p}}_{p,1}}\|b\|_{L^\infty_T\dot B^{\frac{d}{p}+2}_{p,1}}+T\|b\|^2_{L^\infty_T\dot B^{\frac{d}{p}+1}_{p,1}}).
\end{align*}
The above inequality combined with Proposition \ref{localwell} shows that
\begin{align}
I^S_3
\le&C2^{-2N}(\ln N)^{-1}N^{1-\alpha}((\ln N)^{-\frac{1}{2}}2^{2N}+(\ln N)^{-1}2^{2N}
+(\ln N)^{-1}N^{2(1-\alpha)})\notag\\
\le&C(\ln N)^{-\frac{3}{2}}N^{1-\alpha}.\label{Is3}
\end{align}
Collecting \eqref{Is1}-\eqref{Is3} together yields that
\begin{equation}\label{Is}
\begin{aligned}
\|I^{\rm S}\|_{\dot B^{0}_{\infty,q}}\le C(\ln N)^{-2}N^{1-\alpha}+C(\ln N)^{-\frac{3}{2}}N^{1-\alpha}\le C(\ln N)^{-\frac{3}{2}}N^{1-\alpha}.
\end{aligned}
\end{equation}

\textbf{Estimates of $\|b(T)\|_{\dot B^{\frac{d}{p}}_{p,q}}$}
In view of \eqref{IB} and \eqref{Is}, for $\frac{1}{q}<\alpha<1$, we can deduce that
\begin{align*}
\|b(T,\Phi(x,T))\|_{\dot B^{0}_{\infty,q}}\ge &\|I^B\|_{\dot B^{0}_{\infty,q}}-\|b_0\|_{\dot B^{0}_{\infty,q}}-\|I^S\|_{\dot B^{0}_{\infty,q}}\\
\ge& C(\ln N)^{-1}(\ln\ln N)^{-1-3d}N^{1-\alpha}-CN^{\frac{1-\alpha q}{q}}-C(\ln N)^{-\frac{3}{2}}N^{1-\alpha}\\
\ge&C(\ln N)^{-1}(\ln\ln N)^{-1-3d}N^{1-\alpha}.
\end{align*}
Therefore, we obtain from Lemma \ref{uphi}, the above inequality and \eqref{Psi} that
\begin{align}
\|b(x,T)\|_{\dot B^{\frac{d}{p}}_{p,q}}
&\ge C\|b(T,\Phi(x,T))\|_{\dot B^{0}_{\infty,q}}\notag\\
&\ge C(\ln N)^{-1}(\ln\ln N)^{-1-3d}N^{1-\alpha}\to \infty,\,\, \text{as}\,\,N\to\infty.\label{inflation1}
\end{align}
Meanwhile, we can deduce from \eqref{u-1} that
\begin{align}\label{flow field1}
&\|u(x,t)\|_{L^{\infty}_{T}(\dot B^{\frac{d}{p}-1}_{p,1})\cap\mathcal L^1_{T}( \dot B^{\frac{d}{p}+1}_{p,1})}\le  C(\ln\ln N)^{-1}\to 0,\,\, \text{as}\,\,N\to\infty.
\end{align}
Therefore, combining \eqref{initial1}, \eqref{initial2} and \eqref{inflation1}, we complete the proof of Theorem \ref{main} by setting $\delta=(\ln\ln N)^{-1}$ for sufficiently large $N$.

{{Finally, by the construction of initial data $(u_0, b_0)$ in \eqref{initialdata}, we immediately obtain the following corollary which shows that $\dot{B}^{\frac{d}{p}}_{p,q}$ with $q>1$ is not an algebra. More precisely, for $d\ge 1$, $1\le p\le\infty$ and $q>1$, there exist $f$ and $g$ such that
\[\|fg\|_{\dot{B}^{\frac{d}{p}}_{p,q}}\nleq C\|f\|_{\dot{B}^{\frac{d}{p}}_{p,q}}\|g\|_{\dot{B}^{\frac{d}{p}}_{p,q}}.\]
\begin{corollary}
Let $1\le p\le\infty, q>1$ and $\frac{1}{q}<\alpha<1$. Let $N$ be large enough integer. There exist two scalar functions $f^N$ and $g^N$ such that
$$\|f^N\|_{\dot{B}^{\frac{d}{p}}_{p,q}}+\|g^N\|_{\dot{B}^{\frac{d}{p}}_{p,q}}\lesssim {(\ln\ln N)^{-1}}\to 0, ~~\text{as}  ~~N\to\infty,$$
meanwhile,
$$\|{f^N}g^N\|_{\dot{B}^{\frac{d}{p}}_{p,q}}\gtrsim {N^{1-\alpha}}{(\ln\ln N)^{-1}}\to \infty, ~~\text{as}  ~~ N\to\infty.$$
\end{corollary}}
\begin{proof}Let $f^N$ and $g^N$ be defined by
\begin{align*}
f^N=\sum_{\frac{N}{2}\le j\le\frac{4N}{5}}\mathscr{F}^{-1}\Big(\frac{\widehat{\phi}(2^{-j}\xi)}{2^{jd}j^\alpha}\Big),\quad g^N=\mathscr{F}^{-1}\Big(\frac{\widehat{\psi}(2^{-N}\xi)}{2^{Nd}(\ln\ln N)}\Big),
\end{align*}
where  $\widehat{\phi}$ and $\widehat{\psi}$ are consistent with those in \eqref{supp psi} and  \eqref{supp phi}. From \eqref{initial11} and \eqref{initial2}, one easily deduces that
\begin{align*}
\|f^N\|_{\dot{B}^{\frac{d}{p}}_{p,q}}\le CN^{\frac{1}{q}-\alpha}, \quad \|g^N\|_{\dot{B}^{\frac{d}{p}}_{p,q}}\le C(\ln\ln N)^{-1}.
\end{align*}
Therefore, for large enough $N$, due to $\alpha>\frac{1}{q}$, we have
\begin{align*}
\|f^N\|+\|g^N\|_{\dot{B}^{\frac{d}{p}}_{p,q}}\lesssim(\ln\ln N)^{-1}.
\end{align*}
By embedding $\dot B^0_{\infty, q}\hookrightarrow \dot B^{\frac{d}{p}}_{p, q}$ and $\|h\|_{L^\infty}\ge |h(0)|=|\int_{\RR^d}\widehat{h}(\xi)\dd \xi|$, one obtains that
\begin{align*}
&\|f^Ng^N\|_{\dot{B}^{\frac{d}{p}}_{p,q}}\ge C\|f^Ng^N\|_{\dot B^0_{\infty, q}}\ge C\|\dot\Delta_N(f^Ng^N) \|_{L^\infty}\\
\ge&C\Big|\int_{\RR^d}\int_{\RR^d}\widehat{\varphi}(2^{-N}\xi)\widehat{f^N}(\eta)\widehat{g^N}(\xi-\eta)\dd \eta\dd \xi\Big|\\
=&C\Big|\sum_{\frac{N}{2}\le j\le \frac{4N}{5}}\int_{\RR^d}\int_{\RR^d}\widehat{\varphi}(2^{-N}\xi)\frac{\widehat{\phi}(2^{-j}\eta)}{2^{jd}j^\alpha}\frac{\widehat\psi(2^{-N}(\xi-\eta))}{2^{Nd}(\ln\ln N)}\dd \eta\dd \xi\Big|\\
=&C\Big|\sum_{\frac{N}{2}\le j\le\frac{4N}{5}}\int_{\RR^d}\int_{\RR^d}\widehat\varphi(\tilde\xi)\frac{\widehat\phi(\tilde\eta)}{j^\alpha} \frac{\widehat\psi(\tilde\xi-2^{j-N}\tilde\eta)}{\ln\ln N}
\dd\tilde\eta\dd\tilde\xi\Big|.
\end{align*}
We easily deduce from \eqref{1} that
\begin{align*}
&\|f^Ng^N\|_{\dot{B}^{\frac{d}{p}}_{p,q}}\ge C \sum_{\frac{N}{2}\le j\le\frac{4N}{5}} j^{-\alpha}(\ln\ln N)^{-1}\ge CN^{1-\alpha}(\ln\ln N)^{-1}.
\end{align*}
Therefore, we complete this proof.
\end{proof}}
\section*{Acknowledgement}{We greatly appreciate the invaluable comments given by the anonymous referee and the associated editor,
  which are very helpful for improving the paper. }This work is supported by CAEP Foundation (Grant No. CX20210020) and National Natural Science Foundation of China under grant   No. 12071043.

\end{document}